\documentclass[11pt,bezier]{article}
\usepackage{amsmath,amssymb,amsfonts,euscript,graphicx}

\newtheorem{prethm}{{\bf Theorem}}

\newenvironment{thm}{\begin{prethm}{\hspace{-0.5
               em}{\bf.}}}{\end{prethm}}

\newtheorem{prepro}[prethm]{{\bf Theorem}}

\newtheorem{preprop}[prethm]{{\bf Proposition}}

\newtheorem{precor}[prethm]{{\bf Corollary}}

\newenvironment{cor}{\begin{precor}{\hspace{-0.5
               em}{\bf.}}}{\end{precor}}

\newtheorem{preconj}[prethm]{{\bf Conjecture}}

\newtheorem{preremark}[prethm]{{\bf Remark}}

\newenvironment{remark}{\begin{preremark}\rm{\hspace{-0.5
               em}{\bf.}}}{\end{preremark}}

\newtheorem{preexample}[prethm]{{\bf Example}}

\newenvironment{example}{\begin{preexample}\rm{\hspace{-0.5
               em}{\bf.}}}{\end{preexample}}

\newtheorem{prelem}[prethm]{{\bf Lemma}}

\newenvironment{lem}{\begin{prelem}{\hspace{-0.5
               em}{\bf.}}}{\end{prelem}}

\newtheorem{prelam}{{\bf Lemma}}

\newtheorem{preproof}{{\bf Proof.}}

\newenvironment{proof}[1]{\begin{preproof}{\rm
               #1}\hfill{$\Box$}}{\end{preproof}}

\title{\bf \large  More on the Annihilator-Ideal Graph of\\ a Commutative Ring
\thanks
{{\it Key Words}: Annihilator-Ideal Graph; Annihilating-Ideal Graph; Minimal prime ideal; Girth; Star graph.}
\thanks {2010{ \it Mathematics Subject Classification}: 13A15; 13B99; 05C99.}}
\author{{\normalsize  {\sc M.J. Nikmehr }} and\,  {\sc S.M. Hosseini}   \\
{\footnotesize{\it Faculty of Mathematics, K.N. Toosi
University of Technology,}}\\
{\footnotesize{{\it P.O. Box} 16315-1618, \it Tehran, Iran}}\\
{\footnotesize{}}\\
{\footnotesize
  $\mathsf{nikmehr@kntu.ac.ir}$ \quad\quad$\mathsf{smhosseini@mail.kntu.ac.ir}$}
}

\date{}

\begin{document}
\maketitle
\begin{abstract}
{\small \noindent Let $R$ be a commutative ring with identity and
$\Bbb A (R)$ be the set of ideals of $R$ with non-zero annihilator. The  annihilator-ideal graph of $R$,
denoted by $A_{I} (R) $, is a simple graph with the vertex set $\Bbb A(R)^{\ast} := \Bbb A (R) \setminus\lbrace (0) \rbrace $,
and two distinct vertices $I$ and $J$ are adjacent if and only if
$\mathrm{Ann} _{R} (IJ) \neq \mathrm{Ann} _{R} (I) \cup \mathrm{Ann} _{R} (J)$.
 In this paper, we study the affinity between the  annihilator-ideal graph and  the annihilating-ideal graph $\Bbb A \Bbb G (R)$ (a well-known graph with the same vertices and two distinct vertices $I,J$ are adjacent if and only if $IJ=0$) associated with $R$.
All rings
whose $A_{I}(R) \neq \Bbb A \Bbb G (R)$ and
$\mathrm{gr} (A_{I}(R)) =4$ are characterized.
Among other results, we obtain necessary
and sufficient conditions under which  $A_{I} (R)$ is a star graph.
 }
\end{abstract}
\vspace{9mm} \noindent{\bf\large 1. Introduction}\\

\noindent   Many interesting algebraic and combinatorics problems arise when we associate a combinatorics object with an algebraic structure. Therefore, one of the most popular and active area in algebraic combinatorics is study of graphs associated with rings. Papers in this field apply combinatorial
methods to obtain algebraic results in ring theory (see for instance \cite {class}, \cite{Badawi}, \cite{chel} and \cite{yu}).
Moreover, for the most recent study in this direction see  \cite{ab}, \cite{mir}  and  \cite{yu}.

\noindent Throughout this paper, $R$ denotes a unitary commutative ring
which is not an integral domain. The sets of all zero-divisors, nilpotent elements and minimal prime ideals  of $R$ are denoted by $Z(R)$, $\mathrm{Nil}(R)$ and $\mathrm {Min}(R)$, respectively. For a subset $T$ of a ring $R$ we let $T^*=T\setminus\{0\}$. An ideal with non-zero annihilator is called an \textit{annihilating ideal}. The set of annihilating ideals of $R$ is denoted by $\Bbb A(R)$.  The ring $R$ is said to be \textit{reduced} if it has no non-zero nilpotent element.
For any undefined notation or terminology in ring theory, we refer the reader to \cite{ati, lam}.

\noindent Let $G=(V,E)$ be a graph, where $V=V(G)$ is the set of vertices and $E=E(G)$ is the set of edges. If $x,y$ are adjacent vertices, then we write $x-\hspace{-.2cm}-y$. By  $\mathrm{diam}(G)$ and $\mathrm{gr}(G)$, we mean the diameter and the girth of $G$, respectively. A cycle (path) graph of order $n$ is denoted by $C_n$ ($P_n$). A complete bipartite graph with part sizes $m$ and $n$ is denoted by $K_{m,n}$. If the size of one of the parts is $1$, then the graph is said to
be a \textit{star graph}. Also, a complete graph of order $n$ is denoted by $K_n$. The distance between two vertices $x,y$ in $G$ is denoted by $\mathrm{d} _G (x,y)$.
For any $x\in V(G)$, $N_{G}(x)$ represents the set of all adjacent vertices  to $x$.
 For any undefined notation or terminology in graph theory, we refer the reader to \cite{west}.

\noindent The  \textit{annihilator graph} of a ring $R$ is defined as the graph $AG(R)$ with the vertex set $Z(R)^*=Z(R)\setminus\{0\}$, and two distinct vertices $x$ and $y$ are adjacent if and only if $ann_R(xy)\neq ann_R(x)\cup ann_R(y)$. This graph was first introduced and investigated in \cite{Badawi} and many of interesting properties of annihilator graph were studied.
 The \textit{ annihilator-ideal graph of $R$},
denoted by $A_{I}(R) $, is an undirected (simple) graph with the vertex set
$\Bbb A (R)^{\ast} = \Bbb A (R) \setminus \lbrace 0 \rbrace$, and two distinct
vertices $I$ and $J$ are adjacent if and only if
$\mathrm{Ann} _{R} (IJ) \neq \mathrm{Ann} _{R} (I) \cup \mathrm{Ann} _{R} (J)$.
This graph was first introduced and investigated in \cite{sal} and many
of interesting properties of annihilator-ideal graph were studied.
The \textit{annihilating-ideal graph of a ring $R$},
denoted by $\Bbb A \Bbb G (R)$,
 is a graph with the vertex set $\Bbb A (R)^{\ast}$ and two distinct vertices $I$ and $J$
are adjacent if and only if $IJ =(0)$ (see \cite{class, A1, beh1} for more details). It is not hard to see that the
annihilating-ideal graph is a subgraph
of  the  annihilator-ideal graph and so it is interesting to explore some further relations between two
graphs $\Bbb A \Bbb G (R)$ and $A_{I} (R) $.
For instance, it is proved that if $A_{I} (R) \neq \Bbb A \Bbb G (R)$
and $\Bbb A \Bbb G (R)$ is a star graph, then $A_{I} (R)$ is a
complete graph. Among other results, we obtain necessary
and sufficient condition in which $A_{I} (R) = \Bbb A \Bbb G (R)$
and $A_{I} (R)$ is a star graph.


\vspace{9mm} \noindent{\bf\large 2. Preliminars }\vspace{5mm}

\noindent First we recall the fundamental properties of $A_{I} (R) $
that are necessary in this paper.

\noindent The first result of this section has an essential role through the  paper.

\begin{lem}{\rm \cite [Lemma $2.1$]{sal}} \label{lem}
Let $R$ be a ring and $I$ and $J$ be distinct elements of $\Bbb A(R)^{\ast}$. Then the following statements hold.

\noindent {\rm (1)} $I-\hspace{-.2cm}-J$ is not an edge of
$A_{I} (R)$ if and only if $\mathrm{Ann}_{R} (IJ)= \mathrm{Ann}_{R} (I)$
or $\mathrm{Ann}_{R} (IJ)= \mathrm{Ann}_{R} (J)$.

\noindent {\rm (2)} If $I-\hspace{-.2cm}-J$ is an edge of
$\Bbb A \Bbb G (R)$, then $I-\hspace{-.2cm}-J$ is an edge
of $A_{I} (R) $. In particular, if $P$
is a path in $\Bbb A \Bbb G (R)$, then $P$ is a path in $A_{I} (R) $.

\noindent {\rm (3)} If $I-\hspace{-.2cm}-J$ is not an edge of
$A_{I} (R) $, then $\mathrm{Ann}_{R} (I) \subseteq \mathrm{Ann}_{R} (J)$
or $\mathrm{Ann}_{R} (J) \subseteq \mathrm{Ann}_{R} (I)$.

\noindent {\rm (4)} If $\mathrm{Ann}_{R} (I) \nsubseteq \mathrm{Ann}_{R}(J)$ and
$\mathrm{Ann}_{R} (J) \nsubseteq \mathrm{Ann}_{R}(I)$, then $I-\hspace{-.2cm}-J$ is
an edge of $A_{I} (R) $.

\noindent {\rm (5)} If $\mathrm{d} _{\Bbb A \Bbb G (R)} (I,J)=3$,
then $I-\hspace{-.2cm}-J$ is an edge of $A_{I} (R) $.

\noindent {\rm (6)} If $I-\hspace{-.2cm}-J$ is not an edge of
$A_{I} (R) $, then there is a
~$K\in\Bbb A (R)^{\ast} \setminus \lbrace I,J \rbrace$ such that
$I-\hspace{-.2cm}-K-\hspace{-.2cm}-J$ is a path in $\Bbb A \Bbb G (R)$,
and hence $I-\hspace{-.2cm}-K-\hspace{-.2cm}-J$ is also a path in $A_{I} (R) $.
\end{lem}

\noindent By \cite[Theorem 2.1]{beh1}, for every ring $R$,
the annihilating-ideal graph $\Bbb A \Bbb G (R)$ is a connected graph
and $\mathrm{diam} (\Bbb A \Bbb G (R)) \leq 3$. Moreover,
if $\Bbb A \Bbb G (R)$ contains a cycle, then $\mathrm{gr} (\Bbb A \Bbb G (R)) \leq 4$.
By using these facts and  part (6) of Lemma \ref{lem}, we have the following result.

\begin{thm} {\rm \cite [Corollary $2.2$]{sal}}\label{theo}
Let $R$ be a ring with $\vert \Bbb A(R)^{\ast} \vert \geq 2$.
Then $A_{I} (R) $ is a connected graph and
$ \mathrm{diam} (A_{I} (R)) \leq 2$. Moreover,
if $A_{I} (R) $ contains a cycle, then $\mathrm{gr} (A_{I} (R)) \leq 4$.

\end{thm}

\begin{thm} {\rm \cite [Corollary $2.8$]{sal}}  \label{thh1}
Let $R$ be a reduced ring and $A_{I} (R) \neq \Bbb A \Bbb G (R)$. Then $\mathrm{gr}(A_{I} (R))=3$. Furthermore,
there is a cycle $C$ of length three in $A_{I} (R)$
such that each edge of $C$ is not an edge of $\Bbb A \Bbb G (R)$.
\end{thm}

\noindent Next, we provide an example of a non-reduced ring $R$
where $I-\hspace{-.2cm}-J$ is an edge of $A_{I} (R)$
that is not an edge of $\Bbb A \Bbb G (R)$ for some distinct
$I,J \in \Bbb A(R)^{\ast}$, but every path in $A_{I} (R)$
of length two from $I$ to $J$ is also a path in $\Bbb A \Bbb G (R)$.

\begin{example}
Let $R \cong \Bbb Z_{16}$. Clearly,  $\mathfrak{m}=Z(R), \mathfrak{m}^{2}, \mathfrak{m}^{3}$ are non-zero proper ideals of $R$ and $\mathfrak{m}^{4}=(0)$. It is easily seen that $\mathfrak{m}-\hspace{-.2cm}-\mathfrak{m}^{2}$ is an edge of $A_{I} (R)$
that is not an edge of $\Bbb A \Bbb G (R)$.
Moreover, $ \mathfrak{m}-\hspace{-.2cm}-\mathfrak{m}^{3}-\hspace{-.2cm}-\mathfrak{m}^{2}$ is
the only path in  $\Bbb A \Bbb G (R)$ of length two
from $\mathfrak{m}$ to $\mathfrak{m}^{2}$. Indeed, $A_{I} (R) =K_{3}$ and $\Bbb A \Bbb G (R) =K_{1,2}$.
\end{example}

\noindent The following is an example of a non-reduced ring $R$
such that $ A_{I} (R) \neq \Bbb A \Bbb G
(R) $ and if $I-\hspace{-.2cm}-J$ is an edge of $A_{I} (R)$
that is not an edge of $\Bbb A \Bbb G (R)$ for some distinct
$I, J \in \Bbb A (R)^{\ast}$, then there is no path in
$A_{I} (R)$ of length two from $I$ to $J$.

\begin{example}\label{exam}
 Let $R \cong F \times S$, where $F$ is a field and $S$ is a ring with a unique
non-trivial ideal, say $I$.
Clearly, $I_{1}= F \times (0)$, $I_{2} = F \times I$,
$I_{3}= (0) \times I$ and $I_{4}= (0) \times S$
are non-zero proper ideals of $R$.
Then $I_{2}-\hspace{-.2cm}-I_{4}$ is an edge of $A_{I} (R)$
that is not an edge of $\Bbb A \Bbb G (R)$, but there is no path
in $A_{I} (R)$ of length two from $I_{2}$ to $I_{4}$.
Indeed, $\Bbb A \Bbb G (R) \cong P_{4}$ and $A_{I} (R) \cong C_{4}$,
\end{example}

\noindent The next theorem characterizes all rings
$R$ with $A_{I} (R) \neq \Bbb A \Bbb G (R)$ and
$\mathrm{gr} (A_{I}(R)) =4$.

\noindent To prove Theorem \ref{girt}, the following lemma is needed.

\begin{lem} {\rm \cite [Lemma $2.5$]{sal}}\label{ll1}
Let $R$ be a ring and $I, J \in \Bbb A (R)^{\ast}$. Suppose that $I-\hspace{-.2cm}-J$ is an
edge of $A_{I} (R) $ that is not an
edge of $\Bbb A \Bbb G (R)$.
If there is a ~$K \subseteq \mathrm{Ann}_{R}(IJ) \setminus \lbrace I,J \rbrace $
such that $KI \neq (0)$ and $KJ \neq (0)$, then
$ C : I-\hspace{-.2cm}-K-\hspace{-.2cm}-J$ is a path in
$A_{I} (R) $ that is not a path in
$\Bbb A \Bbb G (R)$, and hence $C : I-\hspace{-.2cm}-K-\hspace{-.2cm}-J-\hspace{-.2cm}-I$
is a cycle in $A_{I} (R) $ of length three
and each edge of $C$ is not an edge of $\Bbb A \Bbb G (R)$.
\end{lem}

\begin{thm}\label{girt}
Let $R$ be a ring and $A_{I} (R) \neq \Bbb A \Bbb G (R)$.
Then the following statements are equivalent:

\noindent {\rm (1)} $\mathrm{gr} (A_{I}(R)) =4$;

\noindent {\rm (2)} If $I-\hspace{-.2cm}-J$ is an edge of
$A_{I}(R)$ that is not an edge of
$\Bbb A \Bbb G (R)$ for some distinct $I, J \in \Bbb A(R) ^{\ast}$,
then there is no path in $A_{I}(R)$ of length two from $I$ to $J$;

\noindent {\rm (3)} There are some distinct $I ,J \in \Bbb A(R)^{\ast} $
such that $I-\hspace{-.2cm}-J$ is an edge of $A_{I}(R)$
that is not an edge of $\Bbb A \Bbb G (R)$ and there is no path in
$A_{I}(R)$ of length two from $I$ to $J$;

\noindent {\rm (4)}  $R \cong F\times S$,
where $F$ is a field and $S$ is a ring with a unique
non-trivial ideal.
\end{thm}

\begin{proof}{

\noindent {\rm $(1) \Rightarrow (2)$} Suppose that $I-\hspace{-.2cm}-J$ is an
edge of $A_{I}(R)$ that is not an
edge of $\Bbb A \Bbb G (R)$ for some distinct $I , J \in \Bbb A (R)^{\ast} $.
Since $\mathrm{gr} (A_{I}(R)) \neq 3$,
there is no path in $A_{I}(R)$ of length two from $I$ to $J$.

\noindent {\rm $(2) \Rightarrow (3)$} Since $A_{I} (R) \neq \Bbb A \Bbb G (R)$,
 there are distinct vertices $I , J \in \Bbb A (R)^{\ast}$ such that
$I-\hspace{-.2cm}-J$ is an edge of $A_{I} (R)$
that is not an edge of $\Bbb A \Bbb G (R)$, and hence
there is no path in $A_{I} (R)$ of
length two from $I$ to $J$ by part (2).

\noindent {\rm $(3) \Rightarrow (4)$} Suppose (3) is hold. Then $\mathrm{Ann}_{R} (I) \cap \mathrm{Ann}_{R} (J) = \lbrace 0 \rbrace$.
Since $IJ \neq (0)$ and $\mathrm{Ann}_{R} (I) \cap \mathrm{Ann}_{R} (J) = \lbrace 0 \rbrace$, it follows from Lemma \ref{ll1}
 that $\mathrm{Ann}_{R}(IJ) = \mathrm{Ann}_{R} (I) \cup \mathrm{Ann}_{R} (J) \cup \lbrace I \rbrace $
such that  $I^{2} \neq (0)$ or $\mathrm{Ann}_{R}(IJ) = \mathrm{Ann}_{R} (I) \cup \mathrm{Ann}_{R} (J) \cup \lbrace J \rbrace $ with
$J^{2} \neq (0)$ (We note that if $ \lbrace I,J \rbrace \subseteq \mathrm{Ann}_{R}(IJ)$,
then $A_{I} (R)$ find the path $ I-\hspace{-.2cm}-IJ-\hspace{-.2cm}-J $, a contradiction). Without loss of generality, we may assume that
$\mathrm{Ann}_{R}(IJ) = \mathrm{Ann}_{R} (I) \cup \mathrm{Ann}_{R} (J) \cup \lbrace J \rbrace $
and $J^{2} \neq (0)$. Let $ 0 \neq a \in \mathrm{Ann}_{R} (I)$ and
$0 \neq b \in \mathrm{Ann}_{R} (J)$. Since $\mathrm{Ann}_{R} (I) \cap \mathrm{Ann}_{R} (J) = \lbrace 0 \rbrace$,
we deduce that $a+b \in \mathrm{Ann}_{R} (IJ)$ but $a+b \not\in \mathrm{Ann}_{R} (I)$ and
$a+b \not\in \mathrm{Ann}_{R} (J)$, and hence $R(a+b) =J$.
If $K$ is a non-zero ideal properly contained in $\mathrm{Ann} (I)$,
then $K \subseteq \mathrm{Ann}_{R}(IJ)$, a contradiction. Thus $\mathrm{Ann}_{R} (I)$ is a minimal
ideal of $R$. Similarly, $\mathrm{Ann}_{R} (J)$ is a minimal ideal of $R$.
Since $\mathrm{Ann}_{R}(IJ) = \mathrm{Ann}_{R} (I) \cup \mathrm{Ann}_{R} (J) \cup \lbrace J \rbrace $, where $J^{2} \neq (0)$ and $R(a+b)= J$, we conclude that
$\vert \mathrm{Ann}_{R} (IJ) \vert =3$.
The inclusion relation $J \subseteq \mathrm{Ann}_{R} (IJ)$ implies that $IJ^{2}=(0)$ and so $J^{2} \subseteq \mathrm{Ann}_{R} (I)$. Thus $J^{2} = \mathrm{Ann}_{R} (I)$ (As  $\mathrm{Ann}_{R} (I)$ is a minimal ideal
of $R$). By a similar argument and because of the minimality of $\mathrm{Ann}_{R}(J)$, $IJ= \mathrm{Ann}_{R}(J)$. Since
$\mathrm{Ann} (I)$ and $\mathrm{Ann}(J)$ are two minimal ideals of $R$
and $R(a+b) =J$, $J=\mathrm{Ann}(I)+\mathrm{Ann}(J)$.
Hence $J=J^{2} +IJ$ and so $J^{2}=(J^{2})^{2} +(IJ)^{2}$.
Since $IJ=\mathrm{Ann}(J) \subseteq \mathrm{Ann}(IJ)$,
 $(IJ)^{2} =0$. Therefore, $J^{2}=(J^{2})^{2}$ and so by Brauer's Lemma (see \cite[10.22]{lam}),
$R \cong R_{1} \times R_{2}$, where $R_{1}$ and $R_{2}$ are two rings. To complete the proof we show that one of $R_i$s is a filed and the other one contains exactly one non-trivial ideal. Suppose that $I =(I_{1},I_{2})$ and $J =(J_{1},J_{2})$, where $I_1, J_1$ and $I_2, J_2$ are ideals of $R_1$ and $R_2$, respectively.  With no loss of generality, assume that $I_{1}J_{1}=(0)$ and  $I_{2}J_{2}$ is a (non-zero) minimal ideal of $R_{2}$, i.e., $IJ= (0,I_{2}J_{2})$. As $\mathrm{Ann}_{R}(J)= IJ$, $\mathrm{Ann}_{R} (J) = (\mathrm{Ann}_{R_{1}} (J_{1}),\mathrm{Ann}_{R_{2}}(J_{2}))=(0,I_{2}J_{2})$.
Consequently, $\mathrm{Ann}_{R_{1}}(J_{1})= (0)$ and since $I_{1}J_{1}= (0)$, we conclude that
$I_{1}=(0)$. Thus $J_{1} \neq (0)$, for if not
$ I=(0,I_{2})-\hspace{-.2cm}-(R_{1},0)-\hspace{-.2cm}-J=(0,J_{2}) $ is a path in
$\Gamma ^{\prime} _{\mathrm{Ann}} (R)$ of length two from $I$ to $J$, a contradiction.
Since $I_{2}J_{2}\neq (0)$, $J_{2} \neq (0)$, i.e., $J=(J_{1},J_{2}) \neq (0,0)$
and $J \in \Bbb A (R)^{\ast}$. Hence
$\mathrm{Ann}_{R_{2}} (J_{2}) \neq (0)$ and so $J_{2} \in \Bbb A (R_{2})^{\ast} $.
By the minimality of  $\mathrm{Ann}_{R} (I)=(R_{1}, \mathrm{Ann}_{R_{2}} (I_{2}))$, $\mathrm{Ann}_{R_{2}}(I_{2}) = (0)$.
Also, the equality $\mathrm{Ann}_{R_{2}}(J_{2}) =I_{2}J_{2}$ shows that $I_{2}J_{2}^{2} =(0)$ and hence $J_{2}^{2}= (0)$ (As $\mathrm{Ann}_{R_{2}}(I_{2})= (0)$). Since $IJ \neq (0)$, $I_{2} \neq J_{2}$. Clearly, $J_{2} \subseteq I_{2}J_{2}$ and (again) by the minimality of $I_{2}J_{2}$, we infer $J_{2}=I_{2}J_{2}$ is a minimal ideal of $R_{2}$ and thus $IJ= (0,J_{2})$.
Suppose that $K$ is a non-trivial ideal of $R_{2}$ such that $K \neq J_{2}$ and $KJ_{2} =(0)$.
Now, $(R_{1},K) \subseteq \mathrm{Ann}_{R}(IJ) \setminus \lbrace I,J \rbrace $
and $(R_{1},K)I=(0,KI_{2}) \neq (0,0)$, $(R_{1},K)J= (J_{1},0) \neq (0,0)$ imply that
$ I-\hspace{-.2cm}-(R_{1},K)-\hspace{-.2cm}-J$ is a path in
$A_{I} (R)$ of length two from
$I$ to $J$ by Lemma \ref{ll1}, a contradiction. Thus $K=J_{2}$, i.e.,
$\vert \Bbb A(R_{2})^{\ast} \vert =1$, and so by \cite[Theorem 1.4]{beh1},
$R_{2}$ has exactly one non-trivial ideal.
Since $R_{2}$ is a ring with a unique non-trivial ideal and $IJ \neq (0)$,
we deduce that $I_{2} =R_{2}$.
Next, we claim  that $Z(R_{1}) =\lbrace 0 \rbrace$. Assume to the contrary,
$R_{1}$ contains a non-trivial annihilating-ideal, say $L$.
It is obvious that $LJ_{1} \neq (0)$.
Since $(L,J_{2}) \subseteq \mathrm{Ann}_{R} (IJ) \setminus \lbrace I,J \rbrace $ and $(L,J_{2})I \neq (0,0)$,
$(L,J_{2})J \neq (0,0)$, $ I-\hspace{-.2cm}-(L,J_{2})-\hspace{-.2cm}-J$
is a path in $A_{I} (R)$ of length two from
$I$ to $J$ by Lemma \ref{ll1}, which is impossible and so the claim is proved. Finally, it is enough to show that $R_1$ has no non-trivial ideal.
Assume that $R_{1}$ has a non-trivial ideal $M \neq J_{1}$. Since $(M, J_{2}) \subseteq \mathrm{Ann}(IJ) \setminus \lbrace I, J \rbrace$
and $(M, J_{2})I \neq (0)$, $(M, J_{2})J \neq (0)$,
 one may find the path $ I-\hspace{-.2cm}-(M, J_{2})-\hspace{-.2cm}-J$, a contradiction.  Since $J_{1} \neq (0)$, we conclude that $J_{1} =R_{1}$ and so $R_{1}$ is a field, as desired (We note that
$I =(0, R_{2})$ and $J=(R_{1}, J_{2})$, where $J_{2}$ is the
unique non-trivial ideal of $R_{2}$).

\noindent {\rm $(4) \Rightarrow (1)$} It is clear.
}
\end{proof}

\begin{cor}\label{noneq}
Let $R$ be a ring such that $A_{I} (R) \neq \Bbb A \Bbb G (R)$,
and assume that $R$ is not ring-isomorphic to $F \times S$, where $F$ is
field and $S$ is a ring with a unique non-trivial ideal.
If $E$ is an edge of $A_{I} (R)$
that is not an edge of $\Bbb A \Bbb G (R)$, then $E$ is an edge of a cycle
of length three in $A_{I} (R)$.
\end{cor}

\noindent In view of Theorems \ref{thh1} and \ref{girt}, we have the following corollary.

\begin{cor}\label{cog13}
Let $R$ be a ring such that $A_{I} (R) \neq \Bbb A \Bbb G (R)$.
Then $\mathrm{gr}(A_{I} (R)) \in  \lbrace 3, 4 \rbrace$.
\end{cor}

\vspace{9mm} \noindent{\bf\large 3. When $A_{I} (R)$ and $\Bbb A \Bbb G (R)$ are Identical? }\vspace{5mm}

\noindent We observed that the annihilating-ideal graph $\Bbb A \Bbb G (R)$ is a subgraph of the
 annihilatior-ideal graph $A_{I} (R)$. This section is devoted to characterize all rings whose
 annihilator-ideal graphs are identical to the annihilating-ideal graphs. Also,  annihilator-ideal graphs
and annihilating-ideal graphs associated with reduced and non-reduced rings of girths 4 and $\infty$ are completely identified.

\noindent We first study reduced rings.

\begin{lem}\label{reduced}
Let $R$ be a ring. If $I-\hspace{-.2cm}-J$ is an edge of $A_{I} (R)$,
then $I \cap \mathrm{Ann}_{R} (J) \neq (0)$ and $J \cap \mathrm{Ann}_{R} (I) \neq (0)$.
Moreover, if $R$ is a reduced ring, then the converse is also true.
\end{lem}

\begin{proof}
{Assume that $I-\hspace{-.2cm}-J$ is an edge of $A_{I} (R)$.
Thus $\mathrm{Ann}_{R} (IJ) \neq \mathrm{Ann} _{R} (I) \cup \mathrm{Ann}_{R} (J)$.
Suppose that $K \subseteq \mathrm{Ann}_{R} (IJ) \setminus \mathrm{Ann}_{R} (I) \cup \mathrm{Ann} _{R} (J)$.
Hence $KIJ = (0)$ whereas  $KI \neq (0)$ and $KJ \neq (0)$.
Since $KI \subseteq I$ and $KIJ = (0)$, we deduce that $I \cap \mathrm{Ann}_{R} (J) \neq (0)$. Moreover $KJ \subseteq J$ and $KIJ =(0)$ imply that
$J \cap \mathrm{Ann}_{R} (I) \neq (0)$.
Note that if $KI =J \subseteq I$, then $J^{2} = (0)$, i.e.,
we may suppose that
$R$ is a non-reduced ring. Thus $J \cap \mathrm{Ann}_{R} (J) \neq (0)$.
Since $KI =J \subseteq I$, $I \cap \mathrm{Ann}_{R} (J) \neq (0)$.
Similarly, if $KJ =I$, then $J \cap \mathrm{Ann}_{R} (I) \neq (0)$.

\noindent To prove the other side, let $R$ be a reduced ring,
$I \cap \mathrm{Ann}_{R} (J) \neq (0)$ and $J \cap \mathrm{Ann}_{R} (I) \neq (0)$.
Assume to the contrary, $I-\hspace{-.2cm}-J$ is not an
edge of $A_{I} (R)$.
Hence by part (3) of Lemma \ref{lem},  either $\mathrm{Ann}_{R}(I) \subseteq \mathrm{Ann} (J)$ or
$\mathrm{Ann}_{R} (J) \subseteq \mathrm{Ann}_{R} (I)$.
Without lost of generality, assume that $\mathrm{Ann}_{R}(I) \subseteq \mathrm{Ann} (J)$.
Therefore $\mathrm{Ann} _{R} (I) \cap J \subseteq \mathrm{Ann} _{R} (J) \cap J$
and hence $\mathrm{Ann}_{R} (J) \cap J \neq (0)$, a contradiction. Thus $I-\hspace{-.2cm}-J$
is an edge of $A_{I} (R)$.
}
\end{proof}

\noindent The following example shows that the condition of $R$ to be reduced in Lemma \ref{reduced} is necessary.

\begin{example}\label{neces}
Let $D =\Bbb Z_{2} [X,Y,W]$ and $I = (X^{2},Y^{2},XY,XW)D$ be an ideal of $D$. Let $R =D/ I $.
Let $x =X+I$, $y =Y+I$ and $w=W+I$ be elements of $R$.
Then $\mathrm{Nil}(R) = (x,y)R$ and $Z(R) = (x,y,w)R$ is an ideal of $R$.
It is not hard to check that $(x)R = \mathrm{Ann}_{R} (Z(R))$ and $ \mathrm{Nil}(R)= \mathrm{Ann}_{R} (\mathrm{Nil}(R))$.
Now suppose that $I=\mathrm{Nil} (R)$ and $J =Z(R)$. Thus
$I \cap \mathrm{Ann}_{R}(J) \neq (0)$ and $J \cap \mathrm{Ann}_{R}(I) \neq (0)$, whereas
$\mathrm{Ann}_{R} (IJ) = \mathrm{Ann}_{R}(I) \cup \mathrm{Ann}_{R} (J)$.
Therefore $I-\hspace{-.2cm}-J$ is not an edge of $A_{I} (R)$.
\end{example}

\begin{cor}\label{red1}
Let $R$ be a ring and $I,J \in \Bbb A (R)^{\ast}$. If
$\mathrm{Ann}_{R} (I) \nsubseteq \mathrm{Ann}_{R} (J)$ and
$\mathrm{Ann}_{R} (J) \nsubseteq \mathrm{Ann}_{R} (I)$, then $I-\hspace{-.2cm}-J$ is an edge of
$A_{I} (R)$. Moreover, if $R$ is a reduced ring, then the converts is also true.
\end{cor}

\begin{proof}{
 One side is similar to part (4) of Lemma \ref{lem}.
To prove the other side, assume to the contrary,
$\mathrm{Ann} _{R} (I) \subseteq \mathrm{Ann} _{R} (J)$. Since $R$ is reduced,
$J \cap \mathrm{Ann}_{R} (J) = (0)$.
Thus $J \cap \mathrm{Ann}_{R} (I) \subseteq J \cap \mathrm{Ann}_{R} (J) =(0)$,
a contradiction (see Lemma \ref{reduced}). Thus $\mathrm{Ann}_{R} (I) \nsubseteq \mathrm{Ann}_{R} (J)$.
Similarly, $\mathrm{Ann}_{R} (J) \nsubseteq \mathrm{Ann}_{R} (I)$.
}
\end{proof}

\noindent To prove Theorems \ref{complete} and \ref{Min}, we have to recall the following results.

\begin{lem}\label{remember}
\noindent {\rm(1)} {\rm \cite [Theorem $2.7$]{beh1}} Let $R$ be a reduced ring.
Then $\Bbb A \Bbb G (R)$ is complete graph if and only if
$R \cong F_{1} \times F_{2}$, where $F_{1}$ , $F_{2}$ are fields.

\noindent {\rm(2)} {\rm \cite [Lemma $1.8$]{beh2}} Let $R$ be a reduced ring with finitely many
minimal prime ideals. If $R$ has more than two minimal primes,
then $\mathrm{diam} (\Bbb A \Bbb G(R))=3$.
\end{lem}

\begin{thm}\label{complete}
 Let $R$ be a reduced ring. Then the following statements are equivalent:

\noindent {\rm(i)} $A_{I} (R)$ is a complete graph;

\noindent {\rm(ii)} $\Bbb A \Bbb G (R)$ is a complete graph $($and hence $A_{I} (R) \cong \Bbb A \Bbb G (R) )$;

\noindent {\rm(iii)} $R \cong F_{1} \times F_{2}$, where $F_{1}$, $F_{2}$ are two fields.
\end{thm}

\begin{proof}{ $(i) \Rightarrow (ii)$ Suppose that
$A_{I} (R)$ is a complete graph.
It is enough to show that all elements of $\Bbb A (R)^{\ast}$
are minimal ideals of $R$.
Assume to the contrary, $I \in \Bbb A(R)^{\ast}$ and $J \subsetneqq I$. Thus
$\mathrm{Ann} _{R} (J) \neq (0)$ and clearly $J$ is vertex of $A_{I} (R)$.
By Corollary \ref{red1}, $I-\hspace{-.2cm}-J$ is not an edge of $A_{I} (R)$ and this contradicts the assumption.

\noindent {\rm $(ii) \Rightarrow (iii)$} It follows from part (1) of Lemma \ref{remember}.

\noindent {\rm $(iii) \Rightarrow (i)$} It is clear.
}
\end{proof}

\begin{thm}\label{Min}
Let $R$ be a reduced ring with $3 \leq \vert \mathrm{Min} (R) \vert < \infty $.
Then $A_{I} (R) \neq \Bbb A \Bbb G (R)$ and
$\mathrm{gr} (A_{I} (R))=3 $.
\end{thm}

\begin{proof}
{By part (2) of Lemma \ref{remember},  $\mathrm{diam} (\Bbb A \Bbb G(R))=3$ and
hence $A_{I} (R) \neq \Bbb A \Bbb G (R)$ by Theorem 2.
Since $R$ is a reduced ring and $A_{I} (R) \neq \Bbb A \Bbb G (R)$,  Theorem \ref{thh1} follows that
$\mathrm{gr} (A_{I} (R))=3 $.
}
\end{proof}

\begin{thm} {\rm \cite [Theorem $3.4$]{sal}}\label{indentical}
Let $R$ be a reduced ring.
Then $A_{I} (R) = \Bbb A \Bbb G (R)$
if and only if $ \vert \mathrm{Min} (R) \vert = 2 $.
\end{thm}

\begin{thm}\label{star}
Let $R$ be a reduced ring.
Then the following statements are equivalent:

\noindent {\rm(1)} $\mathrm{gr}(A_{I} (R))=4$;

\noindent {\rm(2)} $A_{I} (R) = \Bbb A \Bbb G (R)$
and $\mathrm{gr}(\Bbb A \Bbb G (R))=4$;

\noindent {\rm(3)} $\mathrm{gr}(\Bbb A \Bbb G (R))=4$;

\noindent {\rm(4)} $ \vert \mathrm{Min} (R) \vert = 2 $;

\noindent {\rm(5)} $\Bbb A \Bbb G (R) \cong K_{m,n} $ with $m,n \geq 2$;

\noindent {\rm(6)} $A_{I} (R) \cong K_{m,n} $ with $m,n \geq 2$.

\end{thm}

\begin{proof}
{ $ (1) \Rightarrow (2)$ Since $\mathrm{gr}(A_{I} (R))=4$,
$A_{I} (R) = \Bbb A \Bbb G (R)$ by Theorem \ref{thh1}.

{\rm $(2) \Rightarrow (3)$} It is clear.

 {\rm $(3) \Rightarrow (4)$} If $ \vert \mathrm{Min} (R) \vert \geq 3 $,
then $\mathrm{gr}(\Bbb A \Bbb G (R))=3$ by \cite [Theorem 7]{A2}, a contradiction. Thus
$ \vert \mathrm{Min} (R) \vert = 2 $.

 {\rm $(4) \Rightarrow (5)$} It is obvious by \cite [Corollary 24]{A1}.

 {\rm $(5) \Rightarrow (6)$} It follows from hypothesis and \cite [Corollary 24]{A1} that $ \vert \mathrm{Min} (R) \vert = 2 $.  By Theorem \ref{indentical},
$A_{I} (R) = \Bbb A \Bbb G (R)$.

 {\rm $(6) \Rightarrow (1)$} It is clear.
}
\end{proof}

\noindent The following example provides a reduced ring with $ \vert \mathrm{Min} (R) \vert = 2 $,
$A_{I} (R) = \Bbb A \Bbb G (R)$ and
$\mathrm{gr}(A_{I} (R)) = \mathrm{gr}(\Bbb A \Bbb G (R))=4$.

\begin{example}\label{poly}
Let $R =K[X,Y]/ (XY)$ where $K$ is an arbitrary field. It is not hard to see that $\Bbb A \Bbb G (R)$
is a complete bipartite graph with two infinite parts
$V_{1} = \lbrace (\overline{f}) \vert~f \in K[X] , f(0)=0 \rbrace $
and $V_{2} = \lbrace (\overline{g}) \vert~g \in K[Y] , g(0)=0 \rbrace $.
Since $R$ is a reduced ring with $ \vert \mathrm{Min} (R) \vert = 2 $,
 by Theorem \ref{star}, $A_{I} (R)$ is a
complete bipartite graph. Thus $A_{I} (R) = \Bbb A \Bbb G (R)$
and $\mathrm{gr}(A_{I} (R)) = \mathrm{gr}(\Bbb A \Bbb G (R))=4$.
\end{example}

\begin{thm}\label{st123}
Let $R$ be a reduced ring.
Then the following statements are equivalent:

\noindent {\rm(1)} $\mathrm{gr}(A_{I} (R)) = \infty$;

\noindent {\rm(2)} $A_{I} (R) = \Bbb A \Bbb G (R)$
and $\mathrm{gr}(\Bbb A \Bbb G (R))= \infty $;

\noindent {\rm(3)} $\mathrm{gr}(\Bbb A \Bbb G (R))= \infty $;

\noindent {\rm(4)} $ \vert \mathrm{Min} (R) \vert = 2 $ and
$\Bbb A \Bbb G (R) \cong K_{1,n}$ for some $n \geq 1$;

\noindent {\rm(5)} $R \cong F \times D $, where $F$ is a field and $D$ is an integral domain;

\noindent {\rm(6)} $A_{I} (R) \cong K_{1,n} $ for some $n \geq 1$.

\end{thm}
\begin{proof}
{ $ (1) \Rightarrow (2)$ Since $\mathrm{gr}(A_{I} (R))= \infty $,
$A_{I} (R) = \Bbb A \Bbb G (R)$ by Theorem \ref{thh1}.

 {\rm $(2) \Rightarrow (3)$} It is obvious.

 {\rm $(3) \Rightarrow (4)$} If $ \vert \mathrm{Min} (R) \vert \geq 3 $,
then $\mathrm{gr}(\Bbb A \Bbb G (R))=3$ by \cite [Theorem 7]{A2}, a contradiction. Thus
$ \vert \mathrm{Min} (R) \vert = 2 $. Since $R$ is reduced and $ \vert \mathrm{Min} (R) \vert = 2 $, $\Bbb A \Bbb G (R)$ is a complete bipartite graph by \cite [Corollary 24]{A1}.
Now, $\mathrm{gr}(\Bbb A \Bbb G (R))= \infty $ implies that
 $\Bbb A \Bbb G (R)$ is a star graph.

{\rm $(4) \Rightarrow (5)$} It follows from \cite [Corollary 2.3 $(2) \Rightarrow (3)$]{beh1}.

{\rm $(5) \Rightarrow (6)$} It is not hard to see that $\Bbb A \Bbb G (R)$ is a star graph,
and so by \cite [Theorem 7]{A2},  $ \vert \mathrm{Min} (R) \vert = 2 $. Hence $A_{I} (R) = \Bbb A \Bbb G (R)$
by Theorem \ref{indentical}, and thus $\mathrm{gr}(A_{I} (R)) = \mathrm{gr}(\Bbb A \Bbb G (R)) = \infty$.
Therefore, $A_{I} (R)$ is a star graph.

 {\rm $(6) \Rightarrow (1)$} It is clear.
}
\end{proof}

\noindent In view of Theorems \ref{star} and \ref{st123}, we  have  the following result.

\begin{cor}\label{final}
Let $R$ be a reduced ring.
Then $A_{I} (R) = \Bbb A \Bbb G (R)$ if and only if
$\mathrm{gr}(A_{I} (R)) = \mathrm{gr}(\Bbb A \Bbb G (R)) \in \lbrace 4, \infty \rbrace $.
\end{cor}

\noindent In the rest of this section, we focus on non-reduced rings for which $\Bbb A \Bbb G (R)$ and
$A_{I} (R)$ are identical.

\begin{thm} \label{t1}
Let $R$ be a non-reduced ring such that $Z(R)$ is not an ideal of $R$.
Then $A_{I} (R) \neq \Bbb A \Bbb G (R)$.
\end{thm}

\begin{proof}
{ Since $R$ is a non-reduced ring and $Z(R)$ is not an ideal of $R$,
$\mathrm{diam} (\Bbb A \Bbb G (R)) =3$ by \cite[Theorem 1.3 $(1) \Leftrightarrow (3)$]{beh2}.
So $A_{I} (R) \neq \Bbb A \Bbb G (R)$
by Theorem \ref{theo}.
}
\end{proof}

\begin{thm}\label{th2}
Let $R$ be a non-reduced ring. Then the induced subgraph of
$A_{I} (R)$ on nilpotent
ideals is a complete graph.
\end{thm}

\begin{proof}
{ Suppose that $I$ and $J$ are two distinct nilpotent ideals of $R$
such that $IJ \neq (0)$. Assume to the contrary, $\mathrm{Ann}_{R}(IJ) = \mathrm{Ann}_{R} (I) \cup \mathrm{Ann}_{R} (J)$. Without loss of generality and by part 1 of Lemma \ref{lem},
we may assume that $\mathrm{Ann}_{R}(IJ) = \mathrm{Ann}_{R} (I)$.
Let $n$ be the least positive integer such that $J^{n} =(0)$.
Suppose that $IJ^{k} \neq (0)$ for each $k$, $1 \leq k < n$.
Then $J^{n-1} \in \mathrm{Ann}_{R} (IJ) \setminus \mathrm{Ann}_{R}(I)$,
a contradiction. Assume that $k$ ($1 \leq k < n$)
is the least positive integer such that $IJ^{k} =(0)$.
Therefore $J^{k-1} \in \mathrm{Ann}_{R} (IJ) \setminus \mathrm{Ann}_{R} (I)$,
a contradiction. Thus $\mathrm{Ann}_{R}(IJ) \neq \mathrm{Ann}_{R} (I) \cup \mathrm{Ann}_{R} (J)$, and hence
$I,J$ are adjacent.}
\end{proof}

\noindent To prove Theorem \ref{prin}, the following lemma is needed.

\begin{lem}\label{non1}
Let $R$ be a non-reduced ring and suppose that $(\mathrm{Nil}(R))^{2} = (0)$. Then the induced subgraph
of $\Bbb A \Bbb G (R)$ on nilpotent ideals is a complete graph.
Moreover, if $R$ is not a principal ideal ring, then the converse is also true.
\end{lem}

\begin{proof}
{Let $(\mathrm{Nil}(R))^{2} = (0)$.
Then clearly the induced subgraph
of $\Bbb A \Bbb G (R)$ on nilpotent ideal is complete graph.
Suppose that $R$ is not a principal ideal ring. We need only to show that
$x^{2} = 0$ for each $x \in \mathrm{Nil(R)}^{\ast}$.
Let $x \in \mathrm{Nil(R)}^{\ast}$ and $x^{2} \neq 0$. Suppose that
$n$ be the least positive integer such that $x^{n} =0$.
Thus $n\geq 3$ and hence $x$, $x^{n-1} +x$ are distinct elements
of $\mathrm{Nil(R)}^{\ast}$. Since $(Rx)(R(x^{n-1}+x))=(0)$
and $x^{n}=0$,  $x^{2}=0$, a contradiction. Therefore $x^{2} =0$ for each
$x \in \mathrm{Nil(R)}^{\ast}$.
}
\end{proof}

\begin{remark}
It is known that if $R$ is a non-reduced principal ideal ring, then $(\mathrm{Nil}(R))^n =(0)$, for some positive integer $n$.
Thus $\mathrm{Nil}(R),\ldots,\mathrm{Nil}(R)^{n-1}$ are only nilpotent ideals of
$R$. If $n=2$ or $3$, then the induced subgraph
of $\Bbb A \Bbb G (R)$ on nilpotent ideals is a complete graph.
\end{remark}

\begin{thm}\label{prin}
Let $R$ be a non-reduced ring  that is not a principal ideal ring, and suppose that
$(\mathrm{Nil}(R))^{2} \neq (0)$.
Then $A_{I} (R) \neq \Bbb A \Bbb G (R)$
and $\mathrm{gr}(A_{I} (R)) = 3$.
\end{thm}

\begin{proof}
{ Since $(\mathrm{Nil} (R))^{2} \neq (0)$, $A_{I} (R) \neq \Bbb A \Bbb G (R)$
by Theorem \ref{th2} and Lemma \ref{non1}.
By Corollary \ref{cog13}, $\mathrm{gr}(A_{I} (R))
\in \lbrace 3,4 \rbrace$. Let $R \cong F \times S$,
where $F$ is a field and $S$ is a ring with a unique
non-trivial ideal. Since $R$ is a non-reduced ring
and  $(\mathrm{Nil} (R))^{2} = (0)$, we deduce from Theorem \ref{girt} that $\mathrm{gr}(A_{I} (R)) =3$.
}
\end{proof}

\noindent   Let $R$ be a ring. By \cite[Lemma 1.11]{beh1},
every minimal ideal is a vertex of $\Bbb A \Bbb G (R)$. Using fact, we may state the following lemma.

\begin{lem} \label{mini}
Let $R$ be a ring. If $I$ is a minimal ideal of $R$, then
$N_{\Bbb A \Bbb G(R)} (I) = N _{A_{I}(R)} (I)$.
\end{lem}

\begin{proof}
{ By part 2 of Lemma \ref{lem}, it is enough to show that
$ N _{A_{I}(R)} (I) \subseteq N_{\Bbb A \Bbb G(R)} (I)$.
Let $I-\hspace{-.2cm}-J$ be an edge of $A_{I} (R)$.
Since $I$ is a minimal ideal of $R$, either $I \subsetneqq J$ or $I \cap J = (0)$. If
$I \cap J = (0)$, then $IJ = (0)$ and $I-\hspace{-.2cm}-J$ is
an edge of $\Bbb A \Bbb G (R)$. Now, suppose that $I \subsetneqq J$ and $IJ \neq (0)$.
Thus $IJ =I$, and hence $\mathrm{Ann}_{R} (IJ) = \mathrm{Ann}_{R} (I)$.
Therefore $I-\hspace{-.2cm}-J$ is not an edge of
$A_{I} (R)$ by part 1 of Lemma \ref{lem},
a contradiction. This means that $IJ = (0) $ and so
$N_{\Bbb A \Bbb G(R)} (I) = N _{A_{I}(R)} (I)$.
}
\end{proof}

\noindent The following Theorem was proved in \cite{sal}.

\begin{thm}\label{salehi}
Let $R$ be a non-reduced ring. Then the following statements
are equivalent:

\noindent {\rm(1)}  $\mathrm{gr}(\Bbb A \Bbb G (R))=4$;

\noindent {\rm(2)} $A_{I} (R) = \Bbb A \Bbb G (R) = K_{2,3}$;

\noindent {\rm(3)} $\Gamma (R) =\overline{K_{1,3}}$;

\noindent {\rm (4)} $R$ is ring-isomorphic to either
$\Bbb Z _{2} \times \Bbb Z_{4}$ or $\Bbb Z _{2} \times \Bbb Z_{2} [X] /(X^{2})$.
\end{thm}

\noindent In the proof of the previous Theorem, the authours claim that
$A_{I} (R) = \Bbb A \Bbb G (R) = K_{2,3}$ is equivalent
to either $R=\Bbb Z _{2} \times \Bbb Z_{4}$ or $R=\Bbb Z _{2} \times \Bbb Z_{2} [X] /(X^{2})$.
The claim is not true because it is clearly both
$R=\Bbb Z _{2} \times \Bbb Z_{4}$ and $R=\Bbb Z _{2} \times \Bbb Z_{2} [X] /(X^{2})$
have four  non-trivial annihilating-ideal and by Example \ref{exam},
$A_{I}(R) \cong C_{4}$.

\noindent Now we provide a correct condition for this Theorem and  its  proof.

\noindent  annihilator-ideal graphs of non-reduced rings of girth 4 are identified in the following result.

\begin{thm}\label{grith}
Let $R$ be a non-reduced ring. Then the following statements
are equivalent:

\noindent {\rm(1)} $\mathrm{gr}(A_{I} (R))=4$;

\noindent {\rm(2)} $A_{I} (R) \neq \Bbb A \Bbb G (R)$
and $\mathrm{gr}(A_{I} (R))=4$;

\noindent {\rm(3)} $R \cong F\times S$,
where $F$ is a field and $S$ is a ring with a unique
non-trivial ideal;

\noindent {\rm (4)} $\Bbb A \Bbb G (R) \cong P_{4}$;

\noindent {\rm (5)} $A_{I}(R) \cong C_{4}$.
\end{thm}

\begin{proof}
{ $(1) \Rightarrow (2)$ Assume to the contrary, $A_{I} (R) = \Bbb A \Bbb G (R)$.
Then $\mathrm{gr}(\Bbb A \Bbb G (R))=4$, and $R \cong D \times S$ where $S$ is
a ring with a unique non-trivial ideal and $D$ is an integral domain
which is not a field by \cite[Theorem 3.5]{A3}.
Suppose that $I$ is the unique non-trivial ideal of $S$.
Then $(0,S)$ and $(D,I)$ are two distinct elements of $\Bbb A(R)^{\ast}$
such that $(0,S)(D,I) \neq (0,0)$.
It is not hard to see that $\mathrm{Ann}_{R}((0,S)(D,I)) \neq \mathrm{Ann}_{R}(0,S) \cup \mathrm{Ann}_{R}(D,I)$.
So $(0,S)-\hspace{-.2cm}-(D,I)$ is an edge of $A_{I} (R)$ that is
not an edge of $\Bbb A \Bbb G(R)$, a contradiction.
Thus $A_{I} (R) \neq \Bbb A \Bbb G (R)$.

\noindent {\rm$(2) \Rightarrow (3)$} It is clear by Theorem \ref{girt}.

\noindent {\rm $(3) \Rightarrow (4)$}  By \cite[Theorem 2]{A2}, the proof is clear.

\noindent {\rm $(4) \Rightarrow (5)$} Suppose that $\Bbb A \Bbb G (R) \cong P_{4}$,
and $\Bbb A \Bbb G (R)$ is the path $I-\hspace{-.2cm}-K-\hspace{-.2cm}-L-\hspace{-.2cm}-J$.
Clearly $IL$ and $KJ$ are two vertices of $\Bbb A \Bbb G (R)$.
Since $ILJ =ILK =KJL =KJI = (0)$, we have
$KJ =K$ and $IL = L$, so $L \subsetneqq I$ and $K \subsetneqq J$.
Clearly, $K \nsubseteq J$ and $J \nsubseteq K$, and hence
$R$ has exactly two minimal ideals $K$ and $L$.
By  Lemma \ref{mini},
$N_{\Bbb A \Bbb G(R)} (L) = N _{\Gamma ^{\prime} _{\mathrm{Ann}}(R)} (L)$ and
$N_{\Bbb A \Bbb G(R)} (K) = N _{\Gamma ^{\prime} _{\mathrm{Ann}}(R)} (K)$.
Since $\mathrm{d} _{\Bbb A \Bbb G (R)} (I, J) =3$,
 $I-\hspace{-.2cm}-J$ is an edge of $A_{I} (R)$
by part 5 of Lemma \ref{lem}, i.e., $A_{I} (R) \cong C _{4}$.

\noindent {\rm$(5) \Rightarrow (1)$} It is clear.
}
\end{proof}

\noindent The following result shows that if $\Bbb A \Bbb G (R)$ is a star graph
and $A_{I} (R) \neq \Bbb A \Bbb G (R)$,
then $A_{I} (R) $ is a complete graph.

\begin{thm}\label{thm8}
Let $R$ be a ring and $A_{I} (R) \neq \Bbb A \Bbb G (R)$.
Suppose that $\Bbb A \Bbb G (R)$ is a star graph.
Then the following statements hold:

\noindent {\rm(1)} $R$ is indecomposable.

\noindent {\rm(2)} $A_{I} (R) $ is a complete graph.
\end{thm}

\begin{proof}
{
\noindent {\rm(1)} First we note that $R$ is a non-reduced ring.
If $R$ is reduced, then  \cite[Corollary 26]{A1} implies that $R \cong F \times D$, where $F$ is a field and $D$ is an integral domain. By Theorem \ref{st123},
 $A_{I} (R) = \Bbb A \Bbb G (R)$,
a contradiction. Suppose that $R\cong R_{1} \times R_{2}$, where $R_{i}$ is a
ring, for $1\leq i \leq 2$. Since $\Bbb A \Bbb G (R)$ has no cycle,
we deduce that $R \cong F \times S$, where $F$ is a field and $S$ is a ring with a unique
non-trivial ideal by \cite[Lemma 1]{A2}. Thus $\Bbb A \Bbb G (R) \cong P_{4}$ and this contradicts the assumption.

\noindent {\rm(2)} Suppose that $I$ is the center of $\Bbb A \Bbb G (R)$.
It is not hard to check that $\mathrm{Ann}(J)=I$, for every  $J(\neq I)\in \Bbb A(R)^{\ast}$. Assume that $J-\hspace{-.2cm}-K$
 is an edge of $A_{I} (R)$ that is not an edge of $\Bbb A \Bbb G (R)$, for some  $J,K \in \Bbb A(R)^{\ast} \setminus \lbrace I \rbrace$.
If $\vert \Bbb A(R)^{\ast} \vert =3$, then there is nothing to prove.
So let $\vert \Bbb A(R)^{\ast} \vert >3$.
Obviously, $\mathrm{Ann}(JK)\neq \mathrm{Ann}(J) \cup \mathrm{Ann}(K)$ and $\mathrm{Ann}(J)=\mathrm{Ann}(K)=I$. Hence $\mathrm{Ann}(JK)\neq I$ and so $JK=I$.
The equalities $JKK=JK^{2}=(0)$ imply that $K^{2} =I$.
Since $JKJ=J^{2}K=(0)$,  $J^{2}=I$. Let  $H, H^{\prime} \neq I$
be two arbitrary annihilating ideals. We prove that the following claims:

\textbf{Claim 1.} The equality  $JH=I$ holds. Assume to the contrary, $JH\neq I$.
Since $JJH=J^{2}H$ and $J^{2}=I$, we deduce that $JJH=(0)$ and hence $JH=I$, a contradiction.

\textbf{Claim 2.} The equality $HH^{\prime}=I$ holds.  For if not, $HH^{\prime} \neq I$.
Since $JH^{2}=IH=(0)$, we deduced  that $JH^{2}=(0)$ and so
 $H^{2}=I$.
 Thus $H^{2}H^{\prime}=IH^{\prime}=(0)$ and hence $HH^{\prime} =I$,
 a contradiction.

 Now, it is easy to see that $\mathrm{Ann}(HH^{\prime}) \neq \mathrm{Ann}(H) \cup \mathrm{Ann}(H^{\prime})$.
 Therefore $H-\hspace{-.2cm}-H^{\prime}$ is an edge of $A_{I} (R)$.
This completes the proof.
}
\end{proof}

%


\begin{thm}\label{Artinian}
Let $R$ be an Artinian ring and $\Bbb A \Bbb G (R)$ be a star graph.
Then $A_{I} (R)$ is a complete graph.
\end{thm}

\begin{proof}
{ By \cite[Theorem 2.6]{beh1}, either $R \cong F_{1} \times F_{2}$,
 where $F_{1}$ and $F_{2}$ are fields, or $R$ is a local ring with non-zero
 maximal ideal $\mathfrak{m}$ and one of the following cases holds:

\noindent {\rm(i)} $\mathfrak{m}^{2}=(0)$ and $\mathfrak{m}$ is the only non-zero proper ideal of $R$.

\noindent {\rm(ii)} $\mathfrak{m}^{3}=(0)$, $\mathfrak{m}^{2}$ is the only minimal ideal of $R$ and
for every distinct proper ideals $I_{1},I_{2}$ of $R$ such that $\mathfrak{m}^{2} \neq I_{i}$ ($i=1,2$), $I_{1}I_{2}=\mathfrak{m}^{2}$.

\noindent {\rm(iii)} $\mathfrak{m}^{4}=(0)$, $\mathfrak{m}^{3} \neq (0)$ and
$\Bbb A(R)^{\ast} = \lbrace \mathfrak{m}, \mathfrak{m}^{2},\mathfrak{m}^{3} \rbrace$.

\noindent If $R \cong F_{1} \times F_{2}$, then $\Bbb A \Bbb G (R) =A_{I} (R)=K_{2}$.
 Suppose that $R$ is a local ring with non-zero maximal ideal $\mathfrak{m}$.
 If case (i) holds, then there is nothing to prove.
If case (ii) holds, then for every pair of distinct proper ideals $I_{1}, I_{2}$such that
$\mathfrak{m}^{2} \neq I_{i}$ ($i=1,2$), we have $I_{1}I_{2}=\mathfrak{m}^{2}$ and so
$\mathrm{Ann}_{R} (I_{1}I_{2}) \neq \mathrm{Ann}_{R} (I_{1}) \cup \mathrm{Ann}_{R} (I_{2})$.
Thus $A_{I} (R)$ is a complete graph.
If case (iii) holds, then $\Bbb A(R)^{\ast} = \lbrace \mathfrak{m}, \mathfrak{m}^{2},\mathfrak{m}^{3} \rbrace$
and $\Bbb A \Bbb G (R)$ is the path $\mathfrak{m}-\hspace{-.2cm}-\mathfrak{m}^{3}-\hspace{-.2cm}-\mathfrak{m}^{2}$.
Obviously, $\mathrm{Ann}_{R} (\mathfrak{m}^{3}) \neq \mathrm{Ann}_{R} (\mathfrak{m}) \cup \mathrm{Ann}_{R}(\mathfrak{m}^{2})$ and so $A_{I} (R) \cong K_{3}$.~}
\end{proof}

\begin{remark}\label{rema123}
Let $\Bbb A \Bbb G (R)$ be a star graph and $I$ be the center of
$\Bbb A \Bbb G (R)$. Then it is easily seen that $I$ is a minimal ideal of $R$. Also, it is worthy to mention that
if $R$ is a non-reduced ring,  then $I^{2}=(0)$ (As $R$ is indecomposable, by part 1 of Theorem \ref{thm8}).
\end{remark}

\noindent Finally, we may characterize star annihilator-ideal graphs  of non-reduced rings.

\begin{thm} \label{infinity}
Let $R$ be a non-reduced ring with $\vert \Bbb A(R)^{\ast} \vert \geq 2$.
Then the following statements are equivalent:

\noindent {\rm(1)} $A_{I} (R)$ is a star graph;

\noindent {\rm(2)} $ \mathrm{gr} (A_{I} (R)) = \infty $;

\noindent {\rm(3)} $\Bbb A \Bbb G (R) =A_{I} (R)$ and $\mathrm{gr} (\Bbb A \Bbb G(R)) = \infty $;

\noindent {\rm(4)} $\mathrm{Nil} (R)$ is a prime ideal and either $Z(R) =\mathrm{Nil} (R)$  and
$\vert \Bbb A(R)^{\ast} \vert = 2 $ or
$Z(R) \neq \mathrm{Nil} (R)$ and $\mathrm{Nil} (R)$ is a minimal ideal of $R$;

\noindent {\rm(5)} Either $A_{I} (R) \cong K_{1,1}$ or
$A_{I} (R) \cong K_{1,\infty}$;

\noindent {\rm(6)} Either $\Bbb A \Bbb G (R) \cong K_{1,1}$ or
$\Bbb A \Bbb G (R) \cong K_{1,\infty}$.
\end{thm}

\begin{proof}
{$(1) \Rightarrow (2)$ It is clear.

\noindent {\rm$(2) \Rightarrow (3)$} It follows from Corollary \ref{cog13}.

\noindent {\rm$(3) \Rightarrow (4)$} By Theorems \ref{th2}, \ref{grith} and
 \cite[Theorem 3.4]{A3}, $\Bbb A \Bbb G (R) =A_{I} (R)$ is a star graph and
 $R$ has at most two non-zero nilpotent ideals.
 Suppose that $R$ has exactly one nilpotent ideal.
 It is not hard to see  that, in this case, $\mathrm{Nil}(R)$ is the only
 nilpotent ideal of $R$.  By Remark \ref{rema123}, $\mathrm{Nil}(R)$ is a minimal ideal of $R$ and so $Z(R) \neq \mathrm{Nil} (R)$. We need only to show that $\mathrm{Nil}(R)$
 is a prime ideal of $R$.
 Assume that $I,J \in \Bbb A(R)^{\ast} \setminus \lbrace \mathrm{Nil}(R) \rbrace$.
 Since $\Bbb A \Bbb G (R)$ is a star graph, we conclude that $IJ \neq (0)$.
 We show that $xy\neq 0$, for every $x \in I,~ y \in J$
 where $x,y \notin \mathrm{Nil}(R)^{\ast}$. Assume to the contrary,
 $xy=0$. Thus $(Rx)(Ry)=(0)$. If $Rx \neq Ry$, then
 $\mathrm{Nil}(R)-\hspace{-.2cm}-Rx-\hspace{-.2cm}-Ry-\hspace{-.2cm}-\mathrm{Nil}(R)$
 is a cycle  in $\Bbb A \Bbb G (R)$, a contradiction.
 Thus $Rx=Ry$, and hence $(Rx)^{2}=(0)$ which means that $Rx \subseteq \mathrm{Nil}(R)$ is
 a nilpotent ideal of $R$, a contradiction. Since $\mathrm{Nil}(R)$ is the center
 of $\Bbb A \Bbb G (R)$, $\mathrm{Nil}(R)Z(R)=(0)$ (We note that by \cite[Theorem 2.2]{beh1}, $Z(R)$ is a vertex) and  $xy \neq 0$
 for every $x,y \in Z(R)^{\ast} \setminus (\mathrm{Nil}(R))^{\ast}$, as desired.

 Now, assume that $R$ has exactly two nilpotent ideals, say $I$ and $J$ where $I$
 is the center of $\Bbb A \Bbb G (R)$. If $J$ is a minimal
 ideal, then
 $I-\hspace{-.2cm}-I+J-\hspace{-.2cm}-J-\hspace{-.2cm}-I$ is a cycle
 in $\Gamma ^{\prime} _{\mathrm{Ann}} (R)$ by Theorem \ref{th2}, a contradiction.
If $J$ is not a minimal ideal and $I\nsubseteq J$, then there exists $J^{\prime} \subseteq J$
and so $J^{\prime}$ is a nilpotent ideal of $R$. Therefore
$I-\hspace{-.2cm}-J-\hspace{-.2cm}-J^{\prime}-\hspace{-.2cm}-I$
is a cycle of length three in $A_{I} (R)$ by Theorem \ref{th2},
a contradiction. Hence $I  \subsetneqq J$.
If $x \in \mathrm{Nil}(R)$, then $Rx \subseteq J$ and so
 $J = \mathrm{Nil}(R)$. We claim that $J^{2}\neq(0)$. If $J^{2}=(0)$, then $J \subseteq \mathrm{Ann}(J)$.
 By the proof of  part 2 of Theorem \ref{thm8},
 $\mathrm{Ann}(J)=I$. Therefore $J \subseteq I$, a contradiction and so the claim is proved.
 Since $J^{2}$ is a nilpotent ideal of $R$, either
$J^{2}=I$ or $J^{2}=J$. If $J^{2}=J$, then
by Nakayama$^{,}$s Lemma, $J=(0)$, which is impossible. Consequently,  $J^{2} =I$.
Now we show that $\mathrm{Nil}(R)= Z(R)$.
Suppose that $x \in Z(R) \setminus \mathrm{Nil}(R)$. Thus $Rx$ is an annihilating
ideal of $R$, and hence $Rx \in \Bbb A(R)^{\ast}$. If $J(Rx)=(0)$,
then $I-\hspace{-.2cm}-J-\hspace{-.2cm}-Rx-\hspace{-.2cm}-I$ is a cycle of length three
 in $\Bbb A \Bbb G (R)$, a contradiction. Thus, we may assume that $J(Rx) \neq (0)$.
 Since $J(Rx)$ is a nilpotent ideal of $R$ and $R$ has exactly two
 nilpotent ideals, either $J(Rx)=I$ or
 $J(Rx)=J$. If $JR(x)=J$, then $J^{2}(Rx)^{2}=J^{2}=I$.
Since $J^{2}(Rx)^{2}=I(Rx)^{2}=(0)$ and $I \neq (0)$, we conclude that $J(Rx)=I$. Set $K=I+Rx$. The equality $IK=I^{2}+I(Rx)=(0)$ implies that
 $K$ is a vertex of $\Bbb A \Bbb G (R)$. Clearly, $K \neq I , J$.
Thus $JK=J(I+Rx)=IJ+J(Rx)=I$ and hence $\mathrm{Ann}(JK)= \mathrm{Ann}(I) = Z(R)$.
In addition, by the proof of part 2 of Theorem \ref{thm8},
$\mathrm{Ann}(J)=\mathrm{Ann}(K)=I$. These facts mean that
$\mathrm{Ann}(JK) \neq \mathrm{Ann}(J) \cup \mathrm{Ann}(K)$. Therefore,
$I-\hspace{-.2cm}-J-\hspace{-.2cm}-K-\hspace{-.2cm}-I$ is a cycle of length three
 in $\Bbb A \Bbb G (R)$, a contradiction. So $Z(R)= \mathrm{Nil}(R)$ and since $I, J=\mathrm{Nil}(R)$ are
  nilpotent ideals of $R$, we deduce that $\vert \Bbb A(R)^{\ast} \vert = 2 $.

\noindent {\rm$(4) \Rightarrow (5)$} Suppose that $\mathrm{Nil}(R)$ is a prime ideal.
Let $Z(R) = \mathrm{Nil} (R)$ and
$\vert \Bbb A(R)^{\ast} \vert = 2 $.
Then  $A_{I} (R) \cong K_{1,1}$.
Now, let $Z(R) \neq \mathrm{Nil} (R)$ and $\mathrm{Nil}(R)$ be a minimal ideal of $R$. Since
$\mathrm{Nil}(R)$ is prime, if $xy=0$ where $x,y \neq 0$, then
$x \in (\mathrm{Nil}(R))^{\ast}$ or $y \in (\mathrm{Nil}(R))^{\ast}$.
Therefore if $I,J \in \Bbb A(R)^{\ast} \setminus \mathrm{Nil}(R)$, then $IJ \neq (0)$. Thus
$\Bbb A \Bbb G (R)$ is a star graph with $\mathrm{Nil}(R)$ as the center. It is clear $\mathrm{Ann}(I)=\mathrm{Ann}(J)=\mathrm{Nil}(R)$, for every
 $I,J \neq \mathrm{Nil}(R)$.
Hence $\mathrm{Ann}(IJ)=\mathrm{Ann}(I) \cup \mathrm{Ann}(J)$, for every $I,J \in \Bbb A(R)^{\ast}$
where $I,J \neq \mathrm{Nil}(R)$, i.e., $A_{I} (R)$
is also star graph. If $A_{I} (R) \neq K_{1,\infty}$,
then $R$ has only finitely many ideals, by  \cite [theorem 1.4]{beh1}.
This implies that $R$ is an Artinian ring. Since $\Bbb A \Bbb G(R)$ is a star graph and $\mathrm{Nil}(R)$ is
a minimal prime ideal, we get a contradiction (See \cite[Theorem 2.6]{beh1}).
Therefore, $\Bbb A \Bbb G (R) \cong K_{1,\infty}$ and so the proof is complete.

\noindent {\rm$(5) \Rightarrow (6)$} It is obvious.

\noindent {\rm$(6) \Rightarrow (1)$} Since $\Bbb A \Bbb G (R)$ is a star graph
and $\Bbb A \Bbb G (R) \ncong K_{1,n}$ for some $1 < n < \infty $, $A_{I} (R) = \Bbb A \Bbb G (R)$ by part 2 of Theorem \ref{thm8}.
Thus $A_{I} (R)$ is a star graph.
}
\end{proof}


\noindent The last example of this paper provides a ring $R$ such that $\Bbb A \Bbb G (R) =A_{I} (R)=K_{1,\infty}$.

\begin{example}  Let $R=\Bbb Z_{2} [X,Y]/ (X^{2}, XY)$ and let $x=X+(X^{2}+XY)$,
$y=Y+(X^{2}+XY) \in R$. Then $Z(R)=(x, y)R$, $\mathrm{Nil}(R)=(x)R=\lbrace 0, x \rbrace$,
and $Z(R) \neq \mathrm{Nil}(R)$. It is clear that $\mathrm{Nil}(R)$ is a minimal prime ideal
of $R$ and $\Bbb A \Bbb G (R) =A_{I} (R)=K_{1,\infty}$.
\end{example}

%


{}

\end{document}